\documentclass[11pt]{article}
\usepackage{article}
\author{Torgeir Aamb\o}
\title{Positselski duality in $\infty$-categories}
\date{}

\begin{document}
\color{myblack}
\maketitle

% \begin{center}
%     \vspace{-0.5cm}
%     The following document is in a ``pre''-preprint stage, meaning it is not yet available on the arXiv. In particular, its contents is still being scrutinized and quality-checked. 
% \end{center}

\begin{abstract}
    We introduce the notion of a contramodule over a cocommutative coalgebra in a presentably symmetric monoidal $\infty$-category $\C$, and prove a symmetric monoidal $\infty$-categorical version of Positselski's comodule-contramodule correspondence when the coalgebra is coidempotent. This gives a new perspective on, and a new proof of local duality---in the sense of Hovey--Palmieri--Strickland and Dwyer--Greenlees---whenever $\C$ is stable and compactly generated. We further consider an analog of Positselski's definition of contramodules over topological rings in the $\infty$-categorical setting, and show that the two perspectives on contramodules are equivalent. As examples we describe the categories of $\Kn$-local spectra, $T(n)$-local spectra and the derived complete category of a ring $R$, as categories of contramodules.
\end{abstract}

\renewcommand{\baselinestretch}{0.75}\normalsize
{
  \hypersetup{linkcolor=myblack}
  \tableofcontents
}
\renewcommand{\baselinestretch}{1.0}\normalsize

\section{Introduction}

Let $k$ be a field and a $C$ a cocommutative coalgebra in the abelian category $\Vect_k$. A \emph{comodule} over $C$ is a vector space $V$ together with a coassociative counital map $V\to V\otimes_k C$. These objects were introduced in the seminal paper \cite{eilenberg-moore_65} and are categorically dual to modules over algebras. In the same paper Eilenberg and Moore introduced a further dual to comodules, which they called \emph{contramodules}. These are vector spaces $V$ with a map $\Hom_k(C, V)\to V$ satisfying similar axioms called contra-associativity and contra-unitality. 

While modules and comodules got their fair share of fame throughout the decades following their introduction, contramodules were seemingly lost to history---virtually forgotten---until dug out from their grave of obscurity by Positselski in the early 2000's. Positselski has since developed a considerable body of literature on contramodules, see for example \cite{positselski_2010, positselski_2011, positselski_2016, positselski_2017_contraadjusted, positselski_2020} or the survey paper \cite{positselski_2022_contramodules}. 

In \cite{positselski_2010} Positselski introduced the comodule-contramodule correspondence, which is an adjunction between the category of comodules and the category of contramodules over a cocommutative coalgebra $C$. This correspondence sat several existing duality theories on a common footing, for example Serre--Grothendieck duality and Feigin--Fuchs central charge duality. Positselski also introduced the coderived and contraderived categories of $C$-comodules and $C$-contramodules respectively, and used this to prove a derived comodule-contramodule correspondence of the form 
\[\Der^{co}(\ComodC)\simeq \Der^{contra}(\ContraC),\]
generalizing for example Matlis--Greenlees--May duality and Dwyer--Greenlees duality---see \cite{positselski_2016}. 

The goal of the this paper is to extend the comodule-contramodule correspondence---which we will refer to as Positselski duality---to $\infty$-categories. The canonical references for $\infty$-categories are \cite{lurie_09} and \cite{Lurie_HA}, which we will freely use throughout the paper. 

We also study such Positselski duality in stable $\infty$-categories, which are natural enhancements of triangulated categories. These serve as a natural analog of the derived comodule-contramodule correspondence.

\subsection*{Motivation}

Let us try to both make a motivation for the traditional Positselski duality theory and for the connection to coalgebras in stable $\infty$-categories. 

We let $X$ be a separated noetherian scheme, $U\subset X$ an open subscheme and denote by $Z = X\backslash U$ its closed complement. The derived category of all $\O_X$-modules, $\Der(\O_X)$, has a full subcategory $\Der(X)$ consisting of complexes with quasi-coherent homology. We define $\Der(U)$ similarly. These are all stable $\infty$-categories. The homotopy category $h\Der(X)$ is precisely the more traditional triangulated derived category of $X$. 

Letting $i\:U\to X$ denote the inclusion, we get an induced functor on the derived categories $i^*\: \Der(X)\to \Der(U)$ by pulling back along $i$. This has a fully faithful right adjoint  $i_*\: \Der(U)\to \Der(X),$ which itself has a further right adjoint $i^!\: \Der(X)\to \Der(U).$ The kernels of $i^*$ and $i^!$ determine two equivalent subcategories of $\Der(X)$, the former of which is the full subcategory $\Der_Z(X)\subseteq \Der(X)$ consisting of complexes with homology supported on $Z$. 

Denoting by $j\: Z\to X$ the other inclusion we obtain symmetric monoidal stable recollement
\begin{center}
    \begin{tikzcd}[sep = large]
        \Der_Z(X) 
        \arrow[rd, yshift=2pt, xshift=1pt, "j_!"] 
        \arrow[dd, "\simeq"] 
        && \\
        
        &\Der(X) 
        \arrow[lu, yshift=-2pt, xshift=-1pt, "j^!"] 
        \arrow[ld, yshift=2pt, xshift=-1pt, "j^*", swap] 
        \arrow[r, yshift=5pt, bend left, "i^*"] 
        \arrow[r, yshift=-5pt, bend right, "i^!"] 
        &\Der(U) \arrow[l, "i_*", swap] \\
        
        \Der(X)^\wedge_Z \arrow[ru, yshift=-2pt, xshift=1pt, "j_*", swap]  
        &&                  
    \end{tikzcd}
\end{center}
where $\Der(X)^\wedge_Z$ denotes the quasi-coherent sheaves on a formal open neighborhood of $Z$. As mentioned, the categories on the left are equivalent, and are the kernels of $i^*$ and $i^!$. 

This equivalence does not on the surface have anything to do with comodules or contramodules, so let us fix this. For simplicity we assume that $X = \Spec(\Z)$, such that $\Der(X)\simeq \Der(\Z)$. Any prime $p$ determines a closed subscheme $P$ of $X$. With this setup we can identify 
\[\Der_P(X) \simeq \Der(\Comod_{\Z/p^\infty}) \quad\text{and}\quad\Der(X)^\wedge_P \simeq \Der(\Contra_{\Z/p^\infty}),\] 
where $\Z/p^\infty$ is the $p$-Prüfer coalgebra of $\Z$. It is the Pontryagin dual of the $p$-adic completion of $\Z$, often denoted $\Z_p$. 

\begin{introrm}
    There is a more familiar description of $\Comod_{\Z/p^\infty}$ as the $p$-power torsion objects in $\Mod_\Z$ and $\Contra_{\Z/p^\infty}$ as the $L$-complete objects in $\Mod_\Z$. The above then reduces to the derived version of Grothendieck local duality by Dwyer--Greenlees, see \cite{dwyer-greenlees_2002}, showing that this is a certain version of Positselski duality. In \cite[2.2(1), 2.2(3)]{positselski_2017_abelian} Positselski proves that the derived complete modules also correspond to a suitably defined version of contramodules over an adic ring. For the above example this is precisely the $p$-adic integers $\Z_p$. The comodules over $\Z/p^\infty$ then correspond to discrete $\Z_p$-modules, see \cite[Sec. 1.9, Sec. 1.10]{positselski_2022_contramodules}. 
\end{introrm}

The above motivates the classical co/contra correspondence, so let us now see how we wish to abstract this.  

As $i^*$ is a symmetric monoidal localization the category $\Der_Z(X)$ is a localizing ideal. By \cite[6.8]{rouquier_2008} there is a compact object $F\in \Der(X)$ with homology supported on $Z$ such that $F$ generates $\Der_Z(X)$ under colimits. Now, as $\Der_Z(X)$ is a compactly generated localizing ideal of a compactly generated symmetric monoidal stable $\infty$-category, the right adjoint $j^!\: \Der(X)\to \Der_Z(X)$ is smashing, hence given as $j_! j^!(\1)\otimes_{\Der(X)} (-)$, where $\1$ denotes the unit in $\Der(X)$. In $\Der(X)$ the object $j_! j^! (\1)$ is the fiber of the unit map $\1\to i_* i^* (\1)$. In fact, $i_* i^* (\1)$ is an idempotent commutative algebra in $\Der(X)$, hence the fiber of the unit map, i.e. $j_! j^!(\1)$, is a coidempotent cocommutative coalgebra. 

Using a dual version of Barr--Beck monadicity, see \cref{ssec:dual-barr-beck}, one can prove that 
\[\Der_Z(X)\simeq \Comod_{j_! j^! (\1)}(\Der(X)).\] 
Similarly, there is an equivalence $\Der(X)^\wedge_Z \simeq \Contra_{j_! j^! (\1)}(\Der(X)),$ which, put together gives us an instance of Positselski duality for stable $\infty$-categories:
\[\Comod_{j_! j^! (\1)}(\Der(X)) \simeq \Contra_{j_! j^! (\1)}(\Der(X)).\]
This is a special case of our second main theorem, \cref{introthm:B}, which is an application of the Positselski duality for commutative coalgebras  set up in \cref{introthm:A}.

\subsection*{Overview of results}

As mentioned, the main goal of this paper is to introduce the notion of comodules and contramodules in $\infty$-categories. Our main result is the following. 

\begin{introthm}[{\cref{thm:Positselski-duality-coidempotent}}]
    \label{introthm:A}
    Let $\C$ be a presentably symmetric monoidal $\infty$-category. For any coidempotent cocommutative coalgebra $C$, there are mutually inverse symmetric monoidal equivalences
    \begin{center}
        \begin{tikzcd}[sep = large]
            \Comod_C(\C) \arrow[r, yshift=2pt, "\simeq"] & \Contra_C(\C) \arrow[l, yshift=-2pt]
        \end{tikzcd}
    \end{center}
    given by the free contramodule and cofree comodule functor respectively. 
\end{introthm}

Our main application of this is to give an alternative perspective on local duality, in the sense of \cite{hovey-palmiery-strickland_97} and \cite{barthel-heard-valenzuela_2018}. 

\begin{introthm}[{\cref{thm:local-duality-co-contra}}]
    \label{introthm:B}
    Let $(\C, \K)$ be a pair consisting of a rigidly compactly generated symmetric monoidal stable $\infty$-category $(\C, \otimes, \1)$ and $\K\subseteq \C$ a set of compact objects. If $\Gamma$ denotes the right adjoint to the fully faithful inclusion of the localizing tensor ideal generated by $\K$, i.e. $i\:\C\Ktors:= \Loc_\C^\otimes (\K)\hookrightarrow \C$, then Positselski duality for the cocommutative coalgebra $i\Gamma \1$, recovers the local duality equivalence $\C\Ktors \simeq \C\Kcomp$, in the sense that 
    \[\C\Ktors \simeq \Comod_{i\Gamma \1} \simeq \Contra_{i\Gamma \1} \simeq \C\Kcomp.\]
\end{introthm}

% As an example of why the two theorems above might be interesting, we have the following descriptions of the categories $\Sp_\Kn$ and $\Sp_{T(n)}$ in chromatic homotopy theory. 

% \begin{introcor}
%     If $p$ is a prime, $n$ a non-negative integer, $\Kn$ the associated Morava $K$-theory and $T(n)$ an associated telescope spectrum, then there are equivalences 
%     \[\Sp_\Kn \simeq \Contra_{M_n\S}(\Spn) \text{ and } \Sp_{T(n)}\simeq \Contra_{M_n^f \S}(\Spn^f)\]
%     of symmetric monoidal stable $\infty$-categories. 
% \end{introcor}

The above theorem can be visualized by the following diagram. 

\begin{center}
    \begin{tikzcd}
        & 
        \Mod_{L\1} 
        \arrow[d, xshift=-2pt] 
        \arrow[rdd, dotted, bend left] 
        & \\
        & 
        \C \arrow[u, xshift=2pt] 
        \arrow[ld, yshift=-2pt, xshift=1pt] 
        \arrow[rd, yshift=2pt, xshift=1pt]                  
        & \\
        \Comod_{\Gamma \1} 
        \arrow[ru, yshift=2pt, xshift=-1pt] 
        \arrow[rr, yshift=2pt, "\simeq"] 
        \arrow[ruu, dotted, bend left] 
        &                                                     
        & \Contra_{\Gamma\1} 
        \arrow[lu, yshift=-2pt, xshift=-1pt] 
        \arrow[ll, yshift=-2pt]
    \end{tikzcd}
\end{center}

It is somewhat unsatisfactory that the category of contramodules is based on the coalgebra $\Gamma\1$ and not on the unit $\Lambda\1$ in $\C\Kcomp$. However, we can quite easily fix this by utilizing the fact that the commutative algebra $\Lambda\1$ is always ``topological'', in the sense that it comes equipped with a commutative pro-dualizable structure. In particular, there is a sequence of dualizable objects $V_k$ such that $\Lambda \1 \simeq \lim_k V_k$, and we can define contramodules over such pro-dualizable commutative algebras as well. 

%Positselski has defined a notion of contramodules over classical topological rings, but his definion does not lift well to the $\infty$-categorical setting. One way to define a contramodule over a topological commutative algebra in a symmetric monoidal $\infty$-category generated by dualizable objects, is as 

\begin{introthm}[\cref{thm:contra-is-contra}, \cref{cor:contra_unit_complete}]
    \label{introthm:C}
    Let $\C$ be a symmetric monoidal $\infty$-category generated by dualizable objects. For any cocommutative coalgebra $C\in \C$, there is an equivalence 
    \[\Contra_C(\C) \simeq \Contra_{C^\vee}(\C),\]
    where $C^\vee := \iHom(C,\1)$ is the linear dual. In particular, when $\C$ is stable and compactly generated, then this gives an equivalence $\C\Kcomp \simeq \Contra_{\Lambda \1}(\C)$ for any local duality context $(\C, \K)$. 
\end{introthm}

\section{General preliminaries}

The goal of this section is to introduce comodules and contramodules over an cocommutative coalgebra in some $\infty$-category $\C$. In order to do this we first review some basic facts about coalgebras, monads and comonads. 

We will for the rest of this section work in some fixed presentably symmetric monoidal $\infty$-category $\C$. In other words, $\C$ is a commutative algebra object in $\PrL$, the category of presentable $\infty$-categories and left adjoint functors. In particular, the monoidal product, which we denote by $-\otimes-\:\C\times \C\to\C$ preserves colimits in both variables. We denote the unit of the monoidal structure by $\1$.

\subsection{Coalgebras, monads and comonads}

We denote the category of commutative algebras in $\C$ by $\CAlg(\C)$. These are the coherently commutative ring objects in $\C$. By \cite[2.4.2.7]{Lurie_HA} there is a symmetric monoidal structure on $\C\op$, and we define the category of cocommutative coalgebras in $\C$ to be the category 
\[\cCAlg(\C):= \CAlg(\C\op)\op.\] 
Classical coalgebras will be referred to as \emph{discrete} in order to avoid confusion. 

\begin{proposition}
    The following properties hold for the category $\cCAlg(\C)$. 
    \begin{enumerate}
        \item The forgetful functor 
        \[U\:\cCAlg(\C)\to \C\] 
        is conservative and creates colimits. \label{prop:properties-coalgebras:item1}
        \item The categorical product of two coalgebras $C, D$ is given by the tensor product of their underlying objects $C\otimes D$. \label{prop:properties-coalgebras:item2}
        \item The $\infty$-category $\cCAlg(\C)$ is presentably symmetric monoidal when equipped with the cartesian monoidal structure. In particular, this means that the forgetful functor $U$ is symmetric monoidal. \label{prop:properties-coalgebras:item3}
        \item The forgetful functor $U$ has a lax-monoidal right adjoint 
        \[cf\:\C\to\cCAlg(\C).\] 
        The image of an object $X\in \C$ is called the cofree coalgebra on $X$. \label{prop:properties-coalgebras:item4}
    \end{enumerate}
\end{proposition}
\begin{proof}
    The presentability and creation of colimits by the forgetful functor is proven in \cite[3.1.2]{lurie_2018_ELL1} and \cite[3.1.4]{lurie_2018_ELL1}. The cartesian symmetric monoidal structure on $\cCAlg(\C)$ follows from \cite[3.2.4.7]{Lurie_HA}. The last item follows from the first three together with the adjoint functor theorem, \cite[5.5.2.9]{lurie_09}. 
\end{proof}

Given any $\infty$-category $\D$, the category of endofunctors $\Fun(\D, \D)$ can be given the structure of a monoidal category via composition of functors, see \cite[1.15]{christ_2023}. 

\begin{definition}
    A monad $M$ on $\D$ is an associative algebra in $\Fun(\D, \D)$. Similarly, a comonad $C$ is a coassociative coalgebra in $\Fun(\D, \D)$. 
\end{definition}

\begin{example}
    Any adjunction of $\infty$-categories $F: \D \leftrightarrows \E : G$
    % \begin{center}
    % \begin{tikzcd}[sep = large]
    %     \D \arrow[r, yshift=2pt, "F"] & \mathcal{E} \arrow[l, yshift=-2pt, "G"]
    % \end{tikzcd}
    % \end{center}
    gives rise to a monad $GF\: \D\to \D$ and a comonad $FG\:\mathcal{E}\to \mathcal{E}$. We call these the \emph{adjunction monad} and \emph{adjunction comonad} of the adjunction $F\dashv G$. 
\end{example}

The category $\D$ is left tensored over $\Fun(\D, \D)$ via evaluation of functors. Hence, for any monad $M$ on $\D$ we get a category of left modules over $M$ in $\D$. 

\begin{definition}
    Let $\D$ be an $\infty$-category and $M$ a monad on $\D$. We define the \emph{Eilenberg--Moore category} of $M$ to be the category of left modules $\LMod_M(\D)$. Objects in $\LMod_M(\D)$ are referred to as \emph{modules over M}. 
\end{definition}

\begin{remark}
    Dually, any comonad $C$ on $\D$ gives rise to a category of left comodules over $C$ in $\D$. We also call this the Eilenberg--Moore category of $C$, and denote it by $\LComod_C(\D)$. Its objects are referred to as \emph{comodules} over $C$. 
\end{remark}

Given a monad $M$ on $\D$ we have a forgetful functor 
\[U_M\: \LMod_M(\D)\to \D.\] 
By \cite[4.2.4.8]{Lurie_HA} this functor admits a left adjoint 
\[F_M\: \D\to \LMod_M(\D)\] 
given on objects by $X\longmapsto MX$. We call this the \emph{free module} functor. The adjunction $F_M\dashv U_M$ is called the \emph{free-forgetful} adjunction of $M$. 

\begin{definition}
    An adjunction is said to be \emph{monadic} if it is equivalent to the free-forgetful adjunction $F_M\dashv U_M$ of a monad $M$. A functor $G\: \mathcal{E}\to \D$ is called \emph{monadic} if it is equivalent to the right adjoint $U_M$ for some monadic adjunction. 
\end{definition}

The existence of the free-forgetful adjunction for a monad $M$ implies that any monad is the adjunction monad of some adjunction. However, there can be more than one adjunction $F\dashv G$ such that $M$ is the adjunction monad for this adjunction.

\begin{definition}
    Let $\D$ be an $\infty$-category and $M$ a monad on $\D$. A left $M$-module $B\in \LMod_M(\D)$ is \emph{free} if it is equivalent to an object in the image of $F_M$. The full subcategory of free modules is called the \emph{Kleisli category} of $M$, and is denoted $\LMod_M\fr(\D)$. 
\end{definition}

The free-forgetful adjunction restricts to an adjunction on the Kleisli category: 
\begin{center}
\begin{tikzcd}[sep = large]
    \D \arrow[r, yshift=2pt, "{F_M}"] & \LMod_M\fr(\D) \arrow[l, yshift=-2pt, "{U_M\fr}"].
\end{tikzcd}
\end{center}
By \cite[1.8]{christ_2023} this is the minimal adjunction whose adjunction monad is equivalent to $M$---in particular, the forgetful functor $U_M\fr$ is fully faithful.  

Using Lurie's $\infty$-categorical version of the Barr--Beck theorem we can also identify the free-forgetful adjunction as the maximal adjunction with adjunction monad $M$. 

\begin{theorem}[{\cite[4.7.3.5]{Lurie_HA}}]
    \label{thm:Lurie-BB}
    A functor $G\: \mathcal{E}\to \D$ of $\infty$-categories is monadic if and only if 
    \begin{enumerate}
        \item $G$ admits a left adjoint,
        \item $G$ is conservative, and
        \item the category $\mathcal{E}$ admits colimits of $G$-split simplicial objects, and these are preserved under $G$. 
    \end{enumerate}
\end{theorem}

\begin{remark}
    By definition, if a functor $G\: \mathcal{E}\to \D$ is monadic, then there is an equivalence of categories $\mathcal{E}\simeq \LMod_{GF}(\D)$, where $F$ is the left adjoint of $G$. 
\end{remark}

\begin{definition}
    Dually, given any comonad $C$ on an $\infty$-category $\mathcal{E}$, there is a forgetful functor $U_C\: \LComod_C(\mathcal{E})\to \mathcal{E}$, which admits a right adjoint 
    \[F_C\: \mathcal{E}\to \LComod_C(\mathcal{E}).\] 
    We call this the \emph{cofree comodule functor}, and hence the adjunction $U_C\dashv F_C$ is called the \emph{cofree-forgetful} adjunction of $C$. Any adjunction 
    \begin{center}
    \begin{tikzcd}[sep = large]
        \D \arrow[r, yshift=2pt, "F"] & \mathcal{E} \arrow[l, yshift=-2pt, "G"]
    \end{tikzcd}
    \end{center}
    equivalent to a cofree-forgetful adjunction for some comonad $C$ on $\mathcal{E}$ is said to be \emph{comonadic}. A functor $F\: \D\to \mathcal{E}$ is said to be \emph{comonadic} if it is equivalent to the left adjoint of a comonadic adjunction. 
\end{definition}

\begin{remark}
    The essential image of $F_C$ in $\LComod_C(\mathcal{E})$ determines the Kleisli category $\LComod_C\fr(\mathcal{E})$ of cofree coalgebras. The cofree-forgetful adjunction restricts to an adjunction on cofree comodules, 
    \begin{center}
    \begin{tikzcd}[sep = large]
        \Comod_C\fr(\mathcal{E}) \arrow[r, yshift=2pt, "{U_C\fr}"] & \mathcal{E} \arrow[l, yshift=-2pt, "F_C"]
    \end{tikzcd}
    \end{center}
    which is the minimal adjunction whose adjunction comonad is $C$. 
\end{remark}

\subsection{Comodules and contramodules}
\label{ssec:comodules_and_contramodules}

Recall that we have fixed a presentably symmetric monoidal $\infty$-category $\C$. Let us now construct the monads and comonads of interest for this paper. We want to mention that the paper \cite{hristova-jones-rumynin_2023} has been of influence for this section.

\begin{example}
    \label{ex:algebra-module-monad}
    Let $A\in \CAlg(\C)$ be a commutative algebra object in $\C$. The algebra structure on $A$ induces an algebra structure on the endofunctor $A\otimes(-)\:\C\to \C$, hence it is a monad on $\C$---see \cite[1.15]{christ_2023}. By \cite[1.17]{christ_2023} the Eilenberg--Moore category of this monad is equivalent to the category of modules over $A$ in $\C$. As $A$ is commutative we denote this by $\Mod_A(\C)$. As $\C$ is presentable and the monad $A\otimes(-)$ preserves colimits, there is a right adjoint $\iHom(A,-)\:\C\to \C$. This is a comonad on $\C$. Since these form an adjoint monad-comonad pair, their Eilenberg--Moore categories are equivalent,
    \[\Mod_A(\C)\simeq \LMod_{A\otimes(-)}(\C)\simeq \LComod_{\iHom(A,-)}(\C),\]
    see \cite[V.8.2]{maclane-moerdijk_1994} in the $1$-categorical situation. The $\infty$-categorical version is exactly the same, and follows from the monadicity and comonadicity of the adjunctions. 
\end{example}

\begin{remark}
    \label{rm:internal-adjunction}
    We also mention that the hom-tensor adjunction is an \emph{internal adjunction}, in the sense that there is an equivalence of internal hom-objects 
    \[\iHom(X\otimes Y, Z) \simeq \iHom(X, \iHom(Y, Z)).\]
    This follows from the hom-tensor adjunction together with a Yoneda argument. 
\end{remark}

% Since $\C$ is a symmetric monoidal category, we can restrict our attention to \emph{monoidal monads} and \emph{monoidal comonads}. For us these are exactly the monads (resp. comonads) that appear as the adjunction monad for a monoidal adjunctions, i.e., $F:\C\rightleftarrows \D:G$ such that $F$ (resp. $G$) is symmetric monoidal. Given a monoidal monad $M$, the corresponding Eilenberg--Moore category $\LMod_M(\C)$ inherits the structure of a presentably symmetric monoidal $\infty$-category. Since $\C$ is symmetric monoidal, the tensor product of two free $M$-modules is still free, hence this restricts to a symmetric monoidal structure on $\LMod_M\fr(\C)$. 

The above example changes in an interesting way when replacing the algebra $A$ with a coalgebra $C$. 

\begin{example}
    \label{ex:coalgebra-comonad}
    Let $C\in \cCAlg(\C)$ be a cocommutative coalgebra object in $\C$. The coalgebra structure on $C$ induces a coalgebra structure on the endofunctor $C\otimes(-)\:\C\to\C$, hence it is a comonad on $\C$. By an argument dual to \cite[1.17]{christ_2023} the Eilenberg--Moore category of this comonad is equivalent to the category of comodules over the coalgebra $C$, defined as 
    \[\Comod_C(\C):= (\Mod_{C}(\C\op))\op.\] 
    As the functor $C\otimes (-)$ is accessible, the category $\Comod_C(\C)$ is presentable by \cite[3.8]{ramzi_2024}---see also \cite[2.1.11]{peroux_2020}. 
\end{example}

\begin{example}
    Let $C\in \cCAlg(\C)$ be a cocommutative coalgebra in $\C$. The comonad $C\otimes (-)$ preserves colimits, hence it has a right adjoint $\iHom(C,-)$. The coalgebra structure on $C$, together with the internal adjunction equivalence from \cref{rm:internal-adjunction}, gives the endofunctor $\iHom(C,-)$ an algebra structure. It is an algebra structure and not a coalgebra structure, as the internal hom is contravariant in the first variable. Hence, $\iHom(C,-)$ is a monad on $\C$.
\end{example} 

Notice that the pair $C\otimes(-)\dashv \iHom(C,-)$ is not an adjoint monad-comonad pair; it is an an adjoint comonad-monad pair. The difference is subtle, but it means, in particular, that their Eilenberg--Moore categories might not be equivalent. This possible non-equivalence is the \emph{raison d'être} for contramodules, which we can then define as follows.

% \begin{lemma}
%     For any adjoint comonad-monad pair $C\dashv M$ on a presentable $\infty$-category $\C$, there is an equivalence 
%     \[\Comod\fr_C(\C)\simeq \Mod\fr_M(\C).\]
% \end{lemma}
% \begin{proof}
%     The objects in the categories are the same, given by an underlying object in $\C$ with additional structure. Hence, it is enough to show that the mapping spaces in the two categories are equivalent. Let $X, Y$ be two objects in $\C$. As the cofree comodule functor $F_C$ is right adjoint to the forgetful functor $U_C$, we get $\Hom_{\Comod_C}(F_C X, F_C Y)\simeq \Hom_\C(C (X), Y)$. By the adjunction $C\dashv M$ on $\C$ we get $\Hom_\C(C(X), Y) \simeq \Hom_\C(X, M(Y))$, which gives finally $\Hom_\C(X, M(Y))\simeq \Hom_{\Mod_M}(F_M X, F_M Y)$, as the free module functor $F_M$ is a left adjoint to the forgetful functor $U_M$.
% \end{proof}

% \begin{remark}
%     This has been noted and proved several times in the $1$-categorical literature on monads and comonads, see for example \cite[Theorem 3]{kleiner_1990}. The proof above is by \cite[2.5]{bohm-brzezinski-wisbauer_2009}. 
% \end{remark}

\begin{definition}
    Let $C\in \cCAlg(\C)$ be a cocommutative coalgebra. A \emph{contramodule} over $C$ is a module over the internal hom-monad $\iHom(C,-)\:\C\to \C$. The category of contramodules over $C$ in $\C$ is the corresponding Eilenberg--Moore category, which will be denoted $\ContraC(\C)$. 
\end{definition}

\begin{notation}
    Since we are working in a fixed category $\C$ we will often simply write $\ContraC$ for the category of contramodules, and $\ComodC$ for the category of comodules. 
\end{notation}

\begin{notation}
    We denote the mapping space in $\Comod_C$ by $\Hom_C$ and the mapping space in $\Contra_C$ by $\Hom^C$. Similarly, the forgetful functors will be denoted 
    \[U_C\:\Comod_C\to \C \quad\text{and}\quad  U^C\: \Contra_C\to \C\] 
    respectively, while their adjoints---the cofree and free functors---will be denoted 
    \[C\otimes (-)\:\C\to \Comod_C \quad \text{and}\quad \iHom(C, -)\: \C\to \Contra_C,\] 
    hoping that it is clear from context whether we use them as above or as endofunctors on $\C$.
\end{notation}

The following proposition is standard for monads and comonads, see for example \cite[5.7]{riehl-verity_2015}. 

\begin{proposition}
    \label{prop:creates-colimits-and-limits}
    If $C$ is a cocommutative coalgebra in $\C$, then the forgetful functor $U_C\:\Comod_C\to \C$ creates colimits. Similarly, the forgetful functor $U^C\:\Contra_C\to \C$ creates limits. 
\end{proposition}

\subsection{The dual monoidal Barr--Beck theorem}
\label{ssec:dual-barr-beck}

Lurie's version of the Barr--Beck monadicity theorem, see \cite[Section 4.7]{Lurie_HA}, allows us to recognize monadic functors from simple criteria. Combined with a recognition theorem for when a monoidal monadic functor is equivalent to $R\otimes(-)$ for some commutative ring $R$, Mathew--Neumann--Noel extended the Barr--Beck theorem to a monoidal version. In this short section we prove a categorical dual version of their result. 

Let $F:\C\rightleftarrows \D:G$ be a pair of adjoint functors between presentably symmetric monoidal $\infty$-categories, such that the left adjoint $F$ is symmetric monoidal. This means that the right adjoint $G$ is lax-monoidal, and does in particular preserve algebra objects. There is for any two objects $X\in \C$ and $Y\in \D$, a natural map
\[F(G(Y)\otimes_\C X)\overset{\simeq}\to FG(Y)\otimes_\D F(X) \to Y\otimes_\D F(X)\]
where the first map is by the symmetric monoidality of $F$, and the second is given by the adjunction counit. By the adjunction property, there is an adjoint map 
\[G(Y)\otimes_\C X \to G(Y\otimes_\D F(X)).\]

\begin{definition}
    An adjoint pair $F\dashv G$ as above is said to satisfy the \emph{monadic projection formula} if the map 
    \[G(Y)\otimes_\C X \to G(Y\otimes_\D F(X))\]
    is an equivalence for all $X\in \C$ and $Y\in \D$. 
\end{definition}

We now state the monoidal Barr--Beck theorem of Mathew, Naumann and Noel. 

\begin{theorem}[{\cite[5.29]{mathew-naumann-noel_2017}}]
    \label{thm:monoidal-BB}
    Let $F:\C\rightleftarrows \D:G$ be an adjunction of presentably symmetric monoidal $\infty$-categories, such that the left adjoint $F$ is symmetric monoidal. If, in addition
    \begin{enumerate}
        \item $G$ is conservative,
        \item $G$ preserves arbitrary colimits, and
        \item $F\dashv G$ satisfies the monadic projection formula,
    \end{enumerate}
    then the adjunction is monoidally monadic, and there is an equivalence of monads 
    \[GF \simeq G(\1_\D)\otimes_\C(-).\]
    In particular, there is an equivalence $\D\simeq \Mod_{G(\1_\D)}(\C)$ of symmetric monoidal $\infty$-categories. 
\end{theorem}

\begin{remark}
    Note that this result is stated only for stable $\infty$-categories in \cite{mathew-naumann-noel_2017}, but also holds unstably by a combination of Lurie's $\infty$-categorical Barr--Beck theorem, \cref{thm:Lurie-BB}, together with the fact that the monadic projection formula applied to the unit gives an equivalence of monads by \cite[3.6]{elmanto-kolderup_2020}. 
\end{remark}

There is also a dual version of the classical Barr--Beck theorem, see for example \cite[4.5]{brantner-mathew_2023}. We wish to extend this to a monoidal version. 

Let $F:\C\rightleftarrows \D:G$ pair of adjoint functors between symmetric monoidal categories, such that the right adjoint $G$ is symmetric monoidal. This means that the left adjoint $F$ is op-lax-monoidal, and does in particular preserve coalgebra objects. There is for any two objects $X\in \C$ and $Y\in \D$, a natural map
\[X\otimes_\C G(Y)\to GF(X)\otimes_\C G(Y) \overset{\simeq}\to G(F(X)\otimes_\D Y)\]
where the first map is given by the adjunction unit and the second by the symmetric monoidality of $G$. By the adjunction property, there is an adjoint map 
\[F(X\otimes_\C G(Y)) \to F(X)\otimes_\D Y .\]

\begin{definition}
    An adjoint pair $F\dashv G$ as above is said to satisfy the \emph{comonadic projection formula} if the map 
    \[F(X\otimes_\C G(Y)) \to F(X)\otimes_\D Y\]
    is an equivalence for all $X\in \C$ and $Y\in \D$. 
\end{definition}

\begin{theorem}
    \label{thm:dual-monoidal-BB}
    Let $F:\C\rightleftarrows \D:G$ be an adjunction of presentably symmetric monoidal $\infty$-categories, such that the right adjoint $G$ is symmetric monoidal. If, in addition
    \begin{enumerate}
        \item $F$ is conservative,
        \item $F$ preserves arbitrary limits, and
        \item $F\dashv G$ satisfies the comonadic projection formula,
    \end{enumerate}
    then the adjunction is comonadic, and there is an equivalence of comonads 
    \[FG\simeq F(\1_C)\otimes_\D (-)\] 
    In particular, this gives an equivalence $\C\simeq \Comod_{F(\1_\C)}(\D).$
\end{theorem}

\begin{remark}
    Before the proof, let us explain intuitively why the statement makes sense. The unit $\1_\C$ in a presentably symmetric monoidal $\infty$-category $\C$ is both a commutative algebra and a cocommutative coalgebra. In the above adjunction we have that the right adjoint $G$ is symmetric monoidal, hence its left adjoint $F$ is op-lax monoidal. In particular, it sends coalgebras to coalgebras, meaning that $F(\1_\C)$ is an cocommutative coalgebra in $\D$. By \cref{ex:coalgebra-comonad} tensoring with $F(\1_\C)$ is a comonad, not a monad, as for \cref{thm:monoidal-BB}. 
\end{remark}

\begin{proof}
    By \cite[4.5]{brantner-mathew_2023} the adjunction is comonadic. A dual version of \cite[3.6]{elmanto-kolderup_2020} shows that there is a map of comonads 
    \[\phi\: FG \to F(\1_\C)\otimes_\D (-),\] 
    and consequently an adjunction 
    \begin{center}
        \begin{tikzcd}[sep = large]
            \Comod_{FG}(\D) \arrow[r, yshift=2pt, "\phi_*"] & \Comod_{F(\1_\C)\otimes_\D (-)}(\D) \arrow[l, yshift=-2pt, "\phi^*"]
        \end{tikzcd}
    \end{center}
    By applying the projection formula to the unit $\1_\C$ we get that $\phi$ is a natural equivalence, which means that the adjunction $(\phi_*, \phi^*)$ is an adjoint equivalence. By \cref{ex:coalgebra-comonad} the Eilenberg--Moore category of the comonad $F(\1_\C)\otimes_\D (-)$ is equivalent to the category of comodules over the cocommutative coalgebra $F(\1_\C)$, finishing the proof. 
\end{proof}

\begin{remark}
    \label{rm:monoidal-structure-comodules}
    It would be nice to understand when the above result gives an equivalence of symmetric monoidal categories. For this one would first need a symmetric monoidal structure on $\Comod_C$. In the dual situation of $\Mod_R(\C)$---the category of modules over a commutative algebra $R\in \C$---this is done by Lurie's relative tensor product, see \cite[Section 4.5.2]{Lurie_HA}. But, for such a relative monoidal product to exist on $\Comod_C$ one needs the tensor product in $\C$ to commute with cosifted limits, which is rarely the case. However, as we will see in the following section, we can in fact obtain a symmetric monoidal structure and a symmetric monoidal equivalence in certain situations---for example when the comonad is idempotent. 
\end{remark}

\section{Positselski duality}
\label{sec:positselski-duality}

Classical Positselski duality, usually referred to as the comodule-contramodule correspondence, is an adjunction between comodules and contramodules over a discrete $R$-coalgebra $C$, where $R$ is an algebra over a field $k$. In particular, the categories involved are abelian, which makes some constructions easier. For example, the monoidal structure on $\Mod_R$ induces monoidal structures on $\Comod_C$ via the relative tensor construction---given by a certain equalizer. For $\infty$-categories the relative tensor construction is more complicated, as we need the monoidal structure to behave well with all higher coherencies, as mentioned in \cref{rm:monoidal-structure-comodules}. We can, however, restrict our attention to a certain type of coalgebra, fixing these issues. This also puts us in the setting we are interested in regarding local duality---see \cref{ssec:local-duality}. 

\subsection{Coidempotent coalgebras}
\label{ssec:coidempotent-coalgebras}

We now restrict our attention to a special class of coalgebras, which will be our focus on for the remainder of the paper. 

\begin{definition}
    A cocommutative coalgebra $C\in \cCAlg(\C)$ is said to be \emph{coidempotent} if the comultiplication map $\Delta\: C\to C\otimes C$ is an equivalence. 
\end{definition}

% \begin{remark}
%     \label{rm:coidempotent-implies-separable}
%     Any coidempotent coalgebra is in particular separable, see \cite[1.6(1)]{ramzi_2023} for a formally dual statement. 
% \end{remark}

The first reason for our focus on coidempotent coalgebras is that their categories of comodules inherit a symmetric monoidal structure from $\C$, which is rarely the case for general coalgebras, see \cref{rm:monoidal-structure-comodules}. 

\begin{lemma}
    \label{lm:coidempotent-then-comod-monoidal}
    Let $C$ be a coidempotent cocommutative coalgebra in $\C$. The category of $C$-comodules $\Comod_C$ inherits the structure of a presentably symmetric monoidal $\infty$-category, making the cofree comodule functor a symmetric monoidal smashing colocalization. Furthermore, any $C$-comodule is cofree, meaning that the fully faithful inclusion $\Comod_C\fr \hookrightarrow \Comod_C$ is a symmetric monoidal equivalence. 
\end{lemma}
\begin{proof}
    By a dual version of \cite[4.8.2.4]{Lurie_HA}, together with \cref{ex:coalgebra-comonad}, the category $\Comod_C$ is a colocalization of $\C$. This implies, by a dual version of \cite[4.8.2.7]{Lurie_HA}, that $\Comod_C$ inherits a symmetric monoidal structure from $\C$. A dual version of \cite[4.8.2.10]{Lurie_HA} implies that the inclusion of cofree comodules to all comodules is an equivalence. Finally, as mentioned in \cref{ex:coalgebra-comonad}, the category $\Comod_C$ is presentable, which together with \cref{prop:creates-colimits-and-limits} implies that $\Comod_C$ is presentably symmetric monoidal. 
\end{proof}

\begin{remark}
    \label{rm:unique-structure}
    As stated in \cref{lm:coidempotent-then-comod-monoidal}, any comodule over a coidempotent coalgebra is cofree. This implies that any comodule has a unique comodule structure. In particular, the structure map $M \to C\otimes M$ for any $C$-comodule $M$ is an equivalence, meaning that the induced tensor product on $\Comod_C$ can be interpreted as simply taking the tensor product of the underlying objects in $\C$. More precisely, the monoidal structure on $\Comod_C$ is given by 
    \[M\otimes_C N := C\otimes (M\otimes N),\] 
    but by the uniqueness of comodule structures, this is simply $M\otimes N$ when treated as an object in $\C$. 
    % The existence of a symmetric monoidal structure will follow from the fact that comodules over coidempotent coalgebras admit unique comodule structures. In particular, if $M$ is a comodule, then the structure map $M \to C\otimes M$ is an equivalence, meaning that $M$ is equivalent to the cofree comodule on its underlying object. This means that any comodule over a coidempotent coalgebra is cofree, giving an equivalence 
    % \[\Comod_C \simeq \Comod_C\fr\]
    % between the Eilenberg--Moore category and the Kleisli category of the comonad $C\otimes (-)$ on $\C$. 
\end{remark}

    % Let $M$ and $N$ be two comodules, and define the tensor product $M\otimes_C N$ to be the tensor product of their underlying objects. By the uniqueness of comodule structures, this still results in a comodule structure on the product, as $M \otimes_C N \simeq M \otimes N \otimes C$, which is the cofree comodule on the underlying object. 
    
    % This means, in particular, that the cofree comodule functor $C\otimes (-) \: \C \to \Comod_C$ is symmetric monoidal by \cite[2.2.1.9]{Lurie_HA}. As the endofunctor $C\otimes (-) \: \C \to \C$ is idempotent, and the forgetful functor $U_C\: \Comod_C\to \C$ is fully faithful whenever $C$ is coidempotent, we can conclude that the cofree comodule functor $C\otimes (-)\: \C\to \Comod_C$ is a smashing colocalization of $\C$. 

% The co-bar construction starts at n=0, hence this does not work (as commented by Irakli):

%Their relative tensor product in $\Comod_C$ is defined by the the two sided co-bar construction,
%\[M\otimes_C N := \lim_n (M \otimes C^{\otimes n} \otimes N),\]
%but, as $C$ is coidempotent this is just the object $N\otimes C\otimes M$, which is the cofree comodule on the underlying object of $M\otimes N$. This means that the relative tensor product is defined for all comodules. The unit for the monoidal structure $-\otimes_C-$ is $C$, and the monoidal structure is symmetric monoidal as the monoidal structure in $\C$ is. 

\begin{lemma}
    The symmetric monoidal structure on the category $\Comod_C$ is closed. 
\end{lemma}
\begin{proof}
    As the cofree-forgetful adjunction creates colimits in the category $\Comod_C$, the functor 
    \[-\otimes_C- \: \Comod_C \times \Comod_C \to \Comod_C\] 
    preserves colimits separately in each variable. In particular, for any comodule $M$ the functor $M\otimes_C (-)$ preserves colimits, hence has a right adjoint $\iHom_C(M,-)$ by the adjoint functor theorem, \cite[5.5.2.9]{lurie_09}. This determines a functor 
    \[\iHom_C(-,-)\: \Comod_C\op\times \Comod_C\to \Comod_C\]
    making $\Comod_C$ a closed symmetric monoidal category.  
\end{proof}

\begin{remark}
    This adjunction, being a hom-tensor adjunction, is also internally adjoint in the sense of \cref{rm:internal-adjunction}. Hence we have an equivalence 
    \[\iHom_C(M\otimes_C N, A) \simeq \iHom_C(M, \iHom_C(N, A))\]
    for all comodules $M, N$ and $A$. 
\end{remark}

We know from \cref{lm:coidempotent-then-comod-monoidal} that the cofree comodule functor $C\otimes(-)\: \C\to \Comod_C$ is a smashing colocalization whenever the coalgebra $C$ is coidempotent. We now wish to have a similar statement for the free contramodule functor $\iHom(C,-)\: \C\to \Contra_C.$ Note that it will not be smashing in general, but otherwise it will have the same features. 

\begin{remark}
    \label{rm:contramodule-structure-on-hom-from-comodule}
    Let $M$ be a $C$-comodule and $V$ any object in $\C$. The structure map $\rho_M\: M\to C\otimes M$ induces a $C$-contramodule structure on the internal hom-object $\iHom(M, V)$, via 
    \[\iHom(C, \iHom(M, V))\simeq \iHom(C\otimes M, V)\overset{-\circ \rho_M}\to \iHom(M, V).\]
\end{remark}

\begin{lemma}
    \label{lm:free-contra-monoidal}
    Let $C$ be a coidempotent cocommutative coalgebra in $\C$. The category of $C$-contramodules $\Contra_C$ inherits the structure of a presentably symmetric monoidal $\infty$-category, making the free contramodule functor a symmetric monoidal localization. In particular, all $C$-contramodules are free. 
\end{lemma}
\begin{proof}
    The functor $\iHom(C,-)\: \C \to \C$ is an idempotent functor, as we have
    \[\iHom(C, \iHom(C,-))\simeq \iHom(C\otimes C, -)\simeq \iHom(C,-)\]
    by the internal adjunction property together with the coidempotency of $C$. This means that the forgetful functor $U^C\: \Contra_C\to \C$ is fully faithful by \cite[5.2.7.4]{lurie_09}, and hence that the free contramodule functor $\iHom(C,-)\:\C\to \Contra_C$ is a localization. In particular, $\iHom(C, X) \simeq X$ for any $C$-contramodule $X$, meaning that any $C$-contramodule is free. 
    
    In order to determine that it induces a symmetric monoidal structure on $\Contra_C$ we need to check that the free functor is compatible with the monoidal structure in $\C$, in the sense of \cite[2.2.1.7]{Lurie_HA}. More precisely, we need to show that if a map $V \to V'$ in $\C$ is a $\iHom(C,-)$-equivalence, then also $V\otimes W \to V' \otimes W$ is. By \cite[2.12(3)]{nikolaus_2016} this property holds whenever $\iHom(V,X)\in \Contra_C$ for any $X\in \Contra_C$ and $V\in \C$, which we now show. 
    
    As all $C$-contramodules are free, we let $X = \iHom(C, A)$ for some $A\in \C$ and $V\in \C$. By the hom-tensor adjunction we get 
    \[\iHom(V, \iHom(C, A))\simeq \iHom(C\otimes V, A).\]
    As $C\otimes V$ is a $C$-comodule (it is cofree), the object $\iHom(C\otimes V, A)$ is a $C$-contramodule by \cref{rm:contramodule-structure-on-hom-from-comodule}. Hence, $\iHom(C,-)$ is compatible with the monoidal structure. By \cite[2.2.1.9]{lurie_09} this implies that the free contramodule functor $\iHom(C,-)\: \C\to \Contra_C$ can be given the structure of a symmetric monoidal functor. 
    
    Finally, as $\Contra_C$ is a localization of a presentably symmetric monoidal category by an accessible functor, it is also presentably symmetric monoidal.  
\end{proof}

\begin{remark}
    To be explicit, the symmetric monoidal structure on $\Contra_C$ is given by $\iHom(C, X\otimes Y)$ for two contramodules $X$ and $Y$, where $\otimes$ is the tensor product in $\C$. In other words, it is the free contramodule on the underlying product. 
\end{remark}

\begin{remark}
    The free contramodule functor is, as mentioned, not smashing in general. This failure is recorded precisely in the symmetric monoidal structure not being given just as the underlying tensor-product, but rather as the free comodule on it. There is a special case, however, where this problem goes away. In the case when the coidempotent coalgebra $C$ is \emph{dualizable}, then the functor $\iHom(C,-)$ is given by $C^\vee \otimes (-)$, where $C^\vee$ is the linear dual of $C$. In this case, $C^\vee$ is an idempotent algebra in $\C$, and there is a symmetric monoidal equivalence between the category of $C$-contramodules and the category of $C^\vee$-modules. 
\end{remark}

% We now need to compare the internal hom objects in $\C$ and $\Comod_C$. By the cofree-forgetful adjunction we have $\Hom_C(M, C\otimes A) \simeq \Hom(U_C M, A)$ for any $C$-comodule $M$ and object $A\in \C$. We wish for this to be enhanced to an internal object relationship. 

% \begin{lemma}
%     Let $M$ be a $C$-comodule and $A\in \C$. There is an equivalence
%     \[U_C \iHom_C(M, C\otimes A)\simeq \iHom(U_C M, A)\]
%     as objects in $\C$. 
% \end{lemma}

% We saw in \cref{rm:cofree-comodule-smashing-colocalization} that the cofree comodule functor was a smashing colocalization. We can naturally wonder whether the corresponding free contramodule functor has the same property. It turns out not to be a localization of $\C$, but not a smashing one.

% \begin{lemma}
%     \label{lm:free-contramodule-functor-localization}
%     The free contramodule functor $\iHom(C,-)\: \C\to \Contra_C$ is a localization. 
% \end{lemma}
% \begin{proof}
%     We prove the equivalent statement that the endofunctor $\iHom(C,-)\:\C\to \C$ is an exact idempotent functor. By the adjunction with $C\otimes -$ we have
%     \[\iHom(C, \iHom(C, -))\simeq \iHom(C\otimes C, -) \simeq \iHom(C,-),\]
%     where the last equivalence is due to the idempotency of $C$. The functor is exact as $\Contra_C$ is a stable category. 
% \end{proof}

% \begin{remark}
%     We will see in \cref{cor:free-contra-monoidal-localization} that it is in fact a symmetric monoidal localization. 
% \end{remark}

We can now deduce our main result, namely that Positselski duality is a symmetric monoidal equivalence for coidempotent coalgebras. 

\begin{theorem}[{\cref{introthm:A}}]
    \label{thm:Positselski-duality-coidempotent}
    Let $\C$ be a presentably symmetric monoidal category and $C\in \C$ a coidempotent cocommutative coalgebra. In this situation there are mutually inverse symmetric monoidal functors
    \begin{center}
        \begin{tikzcd}
            \ComodC(\C) \arrow[rr, yshift=2pt, "{\iHom(C, -)}"] && \ContraC(\C) \arrow[ll, yshift=-2pt, "C\otimes(-)"]
        \end{tikzcd}
    \end{center}
    given on the underlying objects by the free contramodule functor and the cofree comodule functor respectively. 
\end{theorem}
\begin{proof}
    By \cref{lm:coidempotent-then-comod-monoidal} every $C$-comodule is cofree. Similarly, by \cref{lm:free-contra-monoidal} every $C$-contramodule is free. Hence we check an equivalence between these. 
    
    Let $A$ be any object in $\C$. Denote by $C\otimes A$ the corresponding cofree comodule and $\iHom(C, A)$ the corresponding free contramodule. A simple adjunction argument, using both the cofree-forgetful adjunction and the hom-tensor adjunctions in $\C$ and $\Comod_C$, shows that there is an equivalence 
    \[\iHom_C(M, C\otimes A)\simeq C\otimes \iHom(U_C M, A)\]
    for any comodule $M$. In other words, the internal comodule hom is determined by the underlying internal hom in $\C$. For $M = C$ we get
    \[C\otimes \iHom(C, A)\simeq \iHom_C(C, C\otimes A)\]
    which is equivalent to $C\otimes A$ as $C$ is the unit in $\Comod_C$. 

    For the other direction we wish to show that 
    \[\iHom(C, C\otimes A) \simeq \iHom(C, A).\] 
    We do this by showing that the cofree-forgetful functor is an internal adjunction, in the sense of \cref{rm:internal-adjunction}. 

    Let $B$ be an arbitrary object in $\C$, and recall our notation $\Hom(-,-)$ for the mapping space in $\C$. By the hom-tensor adjunction in $\C$ we have 
    \[\Hom(B, \iHom(C, C\otimes A)) \simeq \Hom(C\otimes B, C\otimes A).\]
    Both of these are in the image of the forgetful functor $U_C\:\Comod_C\to \C.$ As $U_C$ is fully faithful whenever $C$ is coidempotent, we get 
    \[\Hom(C\otimes B, C\otimes A) \simeq \Hom_C(C\otimes B, C\otimes A),\]
    where we recall that the latter denotes maps of comodules. By the cofree-forgetful adjunction we have 
    \[\Hom_C(C\otimes B, C\otimes A) \simeq \Hom(C\otimes B, A),\]
    which by the hom-tensor adjunction in $\C$ finally gives 
    \[\Hom(C\otimes B, A) \simeq \Hom(B, \iHom(C, C\otimes A)).\]
    Summarizing the equivalences we have 
    \[\Hom(B, \iHom(C, C\otimes A))\simeq \Hom(B, \iHom(C, A)),\] 
    which by a Yoneda argument implies that there is an equivalence of internal hom-objects $\iHom(C, C\otimes A)\simeq \iHom(C, A)$. 

    % Since comodule structures are unique in $\Comod_C$, see \cref{rm:unique-structure}, there is an equivalence $C\otimes U_C M \simeq M$ for any comodule $M$. This means that we can assume our object $A\in \C$ to be in the image of the forgetful functor. There is a comodule structure on $\iHom(U_C M, C\otimes A)$, given by the equivalence  
    % \[ C\otimes \iHom(U_C M, C\otimes A) \iHom_C(M, C\otimes C\otimes A) \simeq \iHom(M, C\otimes A)\]
    % and the comodule structure on $\iHom(M, C\otimes A)$. By uniqueness, this implies that $\iHom(M, C\otimes A) \simeq \iHom_C(M, C\otimes A)$ for any comodule $M$. Which by the earlier uniqueness remark implies
    % \[\iHom(C, C\otimes A) \simeq \iHom_C(C, C\otimes A)\simeq C\otimes \iHom(C, A)\simeq \iHom(C, A),\]
    % showing that the functors form a mutually inverse pair of equivalences. 

    % Since $\Contra_C$ is equivalent to a symmetric monoidal category, it can itself be made into a symmetric monoidal category. We can do this by declaring that the contramodule tensor product $X\otimes^C Y$ of two contramodules $X$ and $Y$, is 
    % \[X\otimes^C Y := \iHom(C, (C\otimes X)\otimes_C (C\otimes Y)).\]
    % By definition this makes $\iHom(C, -) \: \Comod_C \to \Contra_C$ a symmetric monoidal functor. The functor $C\otimes - \: \Contra_C \to \Comod_C$ is symmetric monoidal because $C\otimes (-)$ is a smashing colocalization. 

    We know by \cref{lm:coidempotent-then-comod-monoidal} and \cref{lm:free-contra-monoidal} that the cofree comodule functor and the free contramodule functor are both symmetric monoidal. By the arguments above, we know that the equivalence $\Comod_C\simeq \Contra_C$ is given by the compositions 
    \begin{center}
        \begin{tikzcd}
            \Comod_C 
            \arrow[rr, yshift = 2pt, "U_C"] 
            && 
            \C 
            \arrow[ll, yshift = -2pt, "C\otimes -"] 
            \arrow[rr, yshift = 2pt, "{\iHom(C, -)}"] 
            && 
            \Contra_C 
            \arrow[ll, yshift = -2pt, "U^C"]
        \end{tikzcd}
    \end{center}
    The composition from left to right is an op-lax symmetric monoidal functor, and the composition from right to left is a lax symmetric monoidal functor. Since they are both equivalences they are necessarily also symmetric monoidal. 
\end{proof}

\begin{remark}
    \label{rm:holds-generally-for-separable}
    We do believe that the above result holds more generally for all coseparable cocommutative coalgebras. These are coalgebras where the comultiplication admits a section, rather than being an equivalence as in the case of coidempotent coalgebras. This coseparable generalization of \cref{thm:Positselski-duality-coidempotent} does hold in the $1$-categorical situation. However, it will in general not be a monoidal equivalence, due to the lack of monoidal structures. 
\end{remark}

% \begin{corollary}
%     \label{cor:free-contra-monoidal-localization}
%     Equipped with the symmetric monoidal structure from \cref{thm:Positselski-duality-coidempotent} the free contramodule functor $\iHom(C, -)\: \C\to \Contra_C$ is symmetric monoidal. 
% \end{corollary}
% \begin{proof}
%     By definition we have
%     \[\iHom(C, A)\otimes^C \iHom(C, B) := \iHom(C, (C\otimes \iHom(C, A))\otimes_C (C\otimes \iHom(C, B))),\]
%     which by the inverse equivalence of $C\otimes -$ and $\iHom(C, -)$ is equivalent to 
%     \[\iHom(C, (C\otimes \iHom(C, A))\otimes_C (C\otimes \iHom(C, B)) \simeq \iHom(C, (C\otimes A)\otimes_C (C\otimes B)).\]
%     By construction of the monoidal structure on $\Comod_C$ we have $(C\otimes A)\otimes_C (C\otimes B)\simeq C\otimes (A\otimes B)$, hence we have
%     \[\iHom(C, (C\otimes A) \otimes_C (C\otimes B))\simeq \iHom(C, C\otimes (A \otimes B))\simeq \iHom(C, A\otimes B),\]
%     where the last equivalence follows from the proof of \cref{thm:Positselski-duality-coidempotent}. 
% \end{proof}

% \begin{remark}
%     As the functor $\iHom(C, -)\:\C\to \Contra_C$ is an coidempotent functor, via the forgetful functor $U^C\:\Contra_C\to \C$, the category $\Contra_C$ is in fact a localization of $\C$, as the forgetful functor $U_C\: \Comod_C\to \C$ is fully faithful when $C$ is coidempotent. This means that the monoidal structure on $\Contra_C$ is the one induced from $\C$ through the localization, as expected. 
% \end{remark}

\subsection{Local duality}
\label{ssec:local-duality}

Our main interest for constructing an $\infty$-categorical version of Positselski duality is related to local duality, in the sense of \cite{hovey-palmiery-strickland_97} and \cite{barthel-heard-valenzuela_2018}. In this section we use \cref{thm:Positselski-duality-coidempotent} to construct an alternative proof of \cite[2.21]{barthel-heard-valenzuela_2018}. We first recall the construction of local duality.

Let $(\C, \otimes, \1)$ be a presentably symmetric monoidal $\infty$-category. The tensor product $\otimes$ preserves filtered colimits separately in each variable, which by the adjoint functor theorem (\cite[5.5.2.9]{lurie_09}) means that the functor $A\otimes (-)$ has a right adjoint $\iHom(A, -)$, making $\C$ a closed symmetric monoidal category. From this internal hom-object we get a functor 
\[(-)^\vee = \iHom(-, \1)\: \C\op \to \C,\] 
which we call \emph{the linear dual}. 

\begin{definition}
    An object $A\in \C$ is \emph{compact} if the functor $\Hom(A, -)$ preserves filtered colimits, and it is \emph{dualizable} if the natural map $A^\vee \otimes B \to \Hom(A,B)$ is an equivalence for all $B\in \C$. 
\end{definition}

The category $\C$ is said to be \emph{compactly generated} if the smallest localizing subcategory containing the compact objects is $\C$. 

\begin{definition}
    A \emph{local duality context} is a pair $(\C, \K)$, where $\C$ is a presentably symmetric monoidal stable $\infty$-category compactly generated by dualizable objects, and $\K\subseteq \C$ is a set of compact objects. 
\end{definition}

\begin{construction}
    Let $(\C, \K)$ be a local duality context. We denote the localizing ideal generated by $\K$ by $\C\Ktors := \Loc_\C^\otimes(\K)$. The \emph{right orthogonal complement} of $\C\Ktors$, in other words those objects $A\in \C$ such that $\Hom(X, A) \simeq 0$ for all $X\in \C\Ktors$ is denoted by $\C\Kloc$. By \cite[2.17]{barthel-heard-valenzuela_2018} this category is also a compactly generated localizing subcategory of $\C$. Lastly, we define the category $\C\Kcomp$ to be the right orthogonal complement to $\C\Kloc$. 

    Now, the fully faithful inclusion $i_{\K-\mathrm{tors}}\:\C\Ktors\hookrightarrow \C$ has a right adjoint functor $\Gamma \:\C\to \C\Ktors$---again by the adjoint functor theorem. This means, in particular, that $\Gamma$ is a colocalization. Similarly, the fully faithful inclusion $i_{\K-\mathrm{loc}}\:\C\Kloc\hookrightarrow \C$ has a left adjoint functor $L\: \C\to \C\Kloc$, and the fully faithful inclusion $i_{\K-\mathrm{comp}}\:\C\Kcomp\hookrightarrow \C$ has a left adjoint functor $\Lambda\:\C\to \C\Kcomp$, which are then both localizations by definition. 
\end{construction}

\begin{remark}
    Note that in the paper \cite{barthel-heard-valenzuela_2018} referenced above, they use the term \emph{left orthogonal complement} instead of right. Both of these are used throughout the literature, but we decided on using \emph{right}, as it felt more natural to the author. 
\end{remark}

The following is usually referred to as the local duality theorem, see \cite[3.3.5]{hovey-palmiery-strickland_97} or \cite[2.21]{barthel-heard-valenzuela_2018}. 

\begin{theorem}
    \label{thm:local-duality-co-contra}
    For any local duality context $(\C, \K)$, 
    \begin{enumerate}
        \item the functor $L$ is a smashing localization,
        \item the functor $\Gamma$ is a smashing colocalization, 
        \item there are equivalences of functors $\Gamma \Lambda \simeq \Gamma$ and $\Lambda\Gamma \simeq \Lambda$,  
        \item the functors $\Lambda\circ i_{\K-\mathrm{tors}}$ and $\Gamma \circ i_{\K-\mathrm{comp}}$ are mutually inverse equivalences, and 
        \item the functors $(\Gamma, \Lambda)$, viewed as endofunctors on $\C$, form an adjoint pair. 
    \end{enumerate} 
    In particular, there are equivalences
    \[\C\Ktors\simeq \C\Kcomp\]
    of symmetric monoidal stable $\infty$-categories. 
\end{theorem}

\begin{remark}
    The result will essentially follow from recognizing $(\Gamma, \Lambda)$, viewed as endofunctors on $\C$, as the adjoint comonad-monad pair $C\otimes (-)\dashv \iHom(C,-)$ for a certain coidempotent cocommutative coalgebra $C$, and then applying \cref{thm:Positselski-duality-coidempotent}. 
\end{remark}

\begin{proof}
    Let us first consider the comodule side of the story. By \cite[3.3.3]{hovey-palmiery-strickland_97} the functor $L$ is a smashing localization, as it is a finite localization away from $\K$. By construction the functor $\Gamma$ is determined by the kernel of the localization $A\to LA$, hence it is a smashing colocalization. This proves part $(1)$ and $(2)$. 

    As $\Gamma$ is smashing it is given by $\Gamma A \simeq \Gamma \1 \otimes A$, and as $\C\Ktors$ is an ideal, it inherits a symmetric monoidal structure from $\C$, making $\Gamma$ a symmetric monoidal functor. In particular, the object $\Gamma \1$ is the unit in $\C\Ktors$. The unit in a compactly generated symmetric monoidal stable $\infty$-category is both a commutative algebra and a cocommutative coalgebra. The inclusion $i_{\K-\mathrm{tors}}\: \C\Ktors\hookrightarrow \C$ is op-lax monoidal, as it is the left adjoint of a symmetric monoidal functor, meaning that it preserves coalgebras. In particular, $\Gamma\1$ treated as an object in $\C$ is a cocommutative coalgebra. Since $\Gamma$ is a smashing colocalization $\Gamma \1$ is a coidempotent coalgebra. By \cref{thm:dual-monoidal-BB} there is then an equivalence $\C\Ktors\simeq \Comod_{\Gamma\1}$, meaning that $\Gamma$ is identified with the cofree $\Gamma \1$-comodule functor. 

    Let us now turn to the contramodule side of the story. By defintion, there is a cofiber sequence $\Gamma \1 \to \1 \to L\1$, which induces a cofiber sequence 
    \[\iHom(L\1, A) \to \iHom(\1, A)\simeq A \to \iHom(\Gamma \1, A)\]
    for any $A\in \C$. Letting $X \in \C\Kcomp$ we have, by defintion, that $\Hom(B,X) \simeq 0$ for any $B \in \C\Kloc$. We want to show that also $\iHom(B, X) \simeq 0$ for all $B\in \C\Kloc$. By a Yoneda argument it suffices to show 
    \[\Hom(U, \iHom(B, X))\simeq \Hom(U, 0) \simeq 0\]
    for all $U \in \C$. By the hom-tensor adjunction we have 
    \[\Hom(U, \iHom(B, X)) \simeq \Hom(U\otimes B, X).\]
    As $L$ is smashing, $\C\Kloc$ is a localizing $\otimes$-ideal, implying that $U\otimes B \in \C\Kloc$ for all $U\in \C$ and $B \in \C\Kloc$. In particular, by the defining property of $X\in \C\Kcomp$, we have $\Hom(U\otimes B, X) \simeq 0$, which then finally implies that also the internal hom $\iHom(B, X) \simeq 0$ for all $B \in \C\Kloc$.   
    
    Now, letting $B = L\1$ gives an equivalence $\iHom(L\1, X)\simeq 0$, meaning that $X \simeq \iHom(\Gamma\1, X)$ by the above cofiber sequence. Hence, any object in $\C\Kcomp$ is a free contramodule over the idempotent cocommutative coalgebra $\Gamma \1$, giving an equivalence $\C\Kcomp \simeq \Contra_{\Gamma\1}$. Furthermore, as $\Lambda$ is a localization by definition, we have an equivalence $\Lambda X \simeq \iHom(\Gamma \1, X)$ for all $X\in \C\Kcomp$, which by idempotency implies that the functors themselves are equivalent. 
    
    To summarize: Any object $M$ in $\C\Ktors$ is a cofree $\Gamma\1$-comodule, and the functor $\Gamma$ coincides with the cofree functor; any object $X \in \C\Kcomp$ is a free $\Gamma\1$-contramodule, and the functor $\Lambda$ coincides with the free functor. This proves part $(5)$. 
    
    By \cref{thm:Positselski-duality-coidempotent} we then get equivalences of categories 
    \[\C\Ktors\simeq \Comod_{\Gamma \1}(\C) \simeq \Contra_{\Gamma \1}(\C)\simeq \C\Kcomp\]
    given by the mutually inverse equivalences 
    \begin{center}
        \begin{tikzcd}
            \Comod_{\Gamma\1} \arrow[rr, yshift=2pt, "{\iHom(\Gamma\1,-)}"] && \Contra_{\Gamma \1} \arrow[ll, yshift=-2pt, "\Gamma\1 \otimes (-)"]
        \end{tikzcd}
    \end{center}
    proving part $(4)$.
    
    Finally, part $(3)$. Let $A\in \C$. The equivalences $\Gamma \Lambda A\simeq \Gamma A$ and $\Lambda\Gamma A \simeq \Lambda A$ follow from the equivalences 
    \[\Gamma \1 \otimes \iHom(\Gamma\1, A)\simeq \iHom_{\Gamma\1}(\Gamma\1, \Gamma\1\otimes A) \simeq \Gamma\1 \otimes A\]
    and 
    \[\iHom(\Gamma\1, \Gamma\1\otimes A)\simeq \iHom(\Gamma \1, A),\] 
    as we showed in the proof of \cref{thm:Positselski-duality-coidempotent}. 
\end{proof}

\begin{remark}
    \label{rm:contramodular-BB}
    The author feels that the equivalence $\C\Kcomp\simeq \Contra_{\Gamma\1}$ should be a formal consequence of a ``contramodular'' Barr--Beck theorem, but such a result has so far escaped our grasp. 
\end{remark}

% \begin{remark}
%     If the more general version of Positselski duality mentioned in \cref{rm:holds-generally-for-separable} holds, one could be able to generalize local duality to slightly more exotic situations, where the functors are not localizations. 
% \end{remark}

The motivation for proving local duality in this setup was to have the following visually beautiful description of local duality.

\begin{center}
    \begin{tikzcd}[sep = large]
        & 
        \Mod_{L\1} 
        \arrow[d, xshift=-2pt] 
        \arrow[rdd, dotted, bend left] 
        & \\
        & 
        \C \arrow[u, xshift=2pt] 
        \arrow[ld, yshift=-2pt, xshift=1pt] 
        \arrow[rd, yshift=2pt, xshift=1pt]                  
        & \\
        \Comod_{\Gamma \1} 
        \arrow[ru, yshift=2pt, xshift=-1pt] 
        \arrow[rr, yshift=2pt, "\simeq"] 
        \arrow[ruu, dotted, bend left] 
        &                                                     
        & \Contra_{\Gamma\1} 
        \arrow[lu, yshift=-2pt, xshift=-1pt] 
        \arrow[ll, yshift=-2pt]
    \end{tikzcd}
\end{center}

Here the dotted arrows correspond to taking the right-orthogonal complement. 

\begin{remark}
    In the local duality theorem there is another naturally appearing functor---briefly encountered in the proof of \cref{thm:Positselski-duality-coidempotent}---which is the right adjoint to the inclusion $\C\Kloc\hookrightarrow \C$, given by $V \simeq \iHom(L\1, -)$. As discussed in \cref{ex:algebra-module-monad} it is a comonadic functor, and its category of comodules is equivalent to $\Mod_{L\1}$. We can think of the objects in $\Comod_{\iHom(L\1,-)}$ as ``co-contramodules''. Adding these to the picture gives 
    \begin{center}
        \begin{tikzcd}[sep = large]
            \Mod_{L\1} 
            \arrow[rd, yshift=-2pt, xshift=-1pt] 
            \arrow[rr, "\simeq"] 
            && 
            \Cocontra_{L\1} 
            \arrow[ld, yshift=2pt, xshift=-1pt] 
            \arrow[dd, dotted] 
            \\
            & 
            \C 
            \arrow[lu, yshift=2pt, xshift=1pt]
            \arrow[ru, yshift=-2pt, xshift=1pt] 
            \arrow[ld, yshift=-2pt, xshift=1pt] 
            \arrow[rd, yshift=2pt, xshift=1pt]   
            &
            \\
            \Comod_{\Gamma \1} 
            \arrow[ru, yshift=2pt, xshift=-1pt] 
            \arrow[rr, "\simeq"]
            \arrow[uu, dotted] 
            && 
            \Contra_{\Gamma\1} 
            \arrow[lu, yshift=-2pt, xshift=-1pt]                
            \end{tikzcd}
    \end{center}
    which also makes this story enticingly connected to $4$-periodic semi-orthogonal decompositions and spherical adjunctions---see \cite[Section 2.5]{dyckerhoff-kaparanov-schechtman-soibelman_2024}. 
\end{remark}

\subsection{Contramodules over topological algebras}
\label{sec:pro-algebras}

As mentioned before, it is somewhat unsatisfactory that the local duality categories $\C\Kloc \simeq \Mod_{L\1}$ and $\C\Ktors \simeq \Comod_{\Gamma\1}$ are based on their respective units, while $\C\Kcomp \simeq \Contra_{\Gamma\1}$ depends on the unit in its dual category. We now present a simple way to deal with this conceptual issue. 

\begin{definition}
    Let $\C$ be a symmetric monoidal $\infty$-category generated by dualizable objects. A commutative algebra $R \in \C$ is called \emph{pro-dualizable} if it is the materialization of a commutative algebra $\overline{R}$ in the pro-category of $\C\dual$, i.e., $\overline{R}\in \CAlg(\Pro(\C\dual))$. In other words, there is a pro-tower 
    \[\overline{R} = (\cdots \to R_2 \to R_1 \to R_0)\]
    of dualizable objects, and an equivalence of commutative algebras $R\simeq \lim_k R_k$. 
\end{definition}

\begin{remark}
    Details on pro-categories in the $\infty$-categorical setting can be found in \cite[A.8.1]{lurie_SAG}. 
\end{remark}

This extra structure allows us to define a second notion of contramodules. 

\begin{definition}
    Let $R \simeq \lim_k R_k$ be a pro-dualizable commutative algebra. The functor 
    \[\lim_k (R_k \otimes -)\: \C \to \C\]
    is a monad, and we define the category of \emph{contramodules over $R$} to be its category of modules. In other words, $\Contra_R := \Mod_{\lim_k (R_k \otimes -)}(\C)$.
\end{definition}

\begin{remark}
    Defining contramodules over objects with pro-structures is not a new idea, and Positselski has sucessfully developed a theory for these in the classical setting. Over a field $F$ the linear dual functor $(-)^\vee := \iHom(-, F)$ gives an equivalence between the opposite category of infinite-dimensional $F$-vector spaces and linearly compact vector spaces -- also called pro-finite-dimensional vector spaces. Any coalgebra $C\in \Vect_F$ is the union of its finite-dimensional sub-coalgebras, hence any algebra $A$ in the category of pro-finite-dimensional vector spaces is a projective limit of finite-dimensional algebras. Positselski has defined a notion of contramodules over such $F$-algebras, using a certain infinite summation monad on the category of sets. 
    
    It is not clear to the author how Positselski's ideas could be more directly lifted to the $\infty$-categorical setting, but Positselski's setup is still the inspiration for our definition of contramodules over pro-dualizable algebras, due to the next result. 
\end{remark}

%It is not hard to see that this approach is equivalent to the approach in \cref{ssec:comodules_and_contramodules} in the case when $\C$ is generated by dualizable objects. 
% We prove this in two steps, which together show that any cocommutative coalgebra determines a unique category of contramodules over a commutative topological algebra, and vice versa. 

\begin{theorem}[\cref{introthm:C}]
    \label{thm:contra-is-contra}
    Let $\C$ be a symmetric monoidal $\infty$-category that is generated by dualizable objects, and let $C\in \C$ be a cocommutative coalgebra. The linear dual $\C^\vee := \iHom(C,\1)$ is a pro-dualizable commutative algebra, and there is an equivalence of $\infty$-categories $\Contra_C \simeq \Contra_{C^\vee}$. 
\end{theorem}
\begin{proof}
    The equivalence $\C\simeq \Ind(\C\dual)$ induces an equivalence on their respective categories of coalgebras, hence there is a choice of presentation $\colim_k C_k$ of the coalgebra $C$ that is a cocommutative coalgebra in $\Ind(\C\dual)$. 
    
    As $\C\dual$ is self-dual, together with the equivalence $\Ind(\C^{\mathrm{dual, op}}) \simeq \Pro(\C\dual)\op$ -- see \cite[A.8.1.2]{lurie_SAG} -- we obtain an equivalence 
    \[\cCAlg(\C)\op \simeq \CAlg(\Pro(\C\dual)),\]
    given by the linear dual $\iHom(-,\1)$. In particular, the linear dual of a cocommutative coalgebra $C$ is always a pro-dualizable commutative algebra. 

    The wanted equivalence now follows from the equivalence of functors 
    \[\iHom(C,-) \simeq \iHom(\colim_k C_k, -) \simeq \lim_k \iHom(C_k,-) \simeq \lim_k(C_k^\vee \otimes -),\]
    which in turn produces equivalent Eilenberg--Moore categories. 
\end{proof}

This gives us our wanted description of $\C\Kcomp$ as the category of contramodules over its unit $\Lambda \1$. 

\begin{corollary}
    \label{cor:contra_unit_complete}
    For any local duality context $(\C, \K)$, there is an equivalence of $\infty$-categories $\C\Kcomp \simeq \Contra_{\Lambda \1}(\C)$. 
\end{corollary}

\subsection{Examples}

Our main interest in \cref{thm:local-duality-co-contra} comes from chromatic homotopy theory and derived completion of rings. We will not present comprehensive introductions to these topics here. The interested reader is referred to \cite[text]{barthel-beaudry_19} for details on the former, and \cite[text]{barthel-heard-valenzuela_2020} for the latter. 

\subsubsection*{Chromatic homotopy theory}

The category of spectra, $\Sp$, is the initial presentably symmetric monoidal stable $\infty$-category. Fixing a prime $p$, one can describe chromatic homotopy theory as the study of $p$-local spectra together with a \emph{chromatic filtration}, coming from the height filtration of formal groups. In such a filtration there is a filtration component corresponding to each natural number $n$, which we will refer to as the $n$-th component. There are, at least, two different chromatic filtrations on $\Sp$, and their conjectural equivalence was recently disproven in \cite{burklund-hahn-levy-schlank_23}. For simplicity we will distinguish these two by referring to them as the \emph{compact filtration} and the \emph{finite filtration}. This latter is a bit misleading, as it is not a finite filtration---the word finite corresponds to a certain finite spectrum. The $n$-th filtration component in the compact filtration is controlled by the Morava $K$-theory spectrum $K(n)$, and the $n$-th filtration component in the finite filtration is controlled by the telescope spectrum $T(n)$. 

We denote the $n$-th component of the compact filtration by $\Sp_n$ and the $n$-th component of the finite filtration by $\Sp_n^f$. The different components are related by smashing localization functors $L_{n-1}\:\Sp_n \to \Sp_{n-1}$ and $L_{n-1}^f\: \Sp_n^f \to \Sp_{n-1}^f$ respectively. 

In the light of local duality, the category $\Sp_{n-1}$ is the category of local objects in $\Sp_n$ for a compact object $L_n F(n) \in \Sp_n$. The torsion objects with respect to $L_n F(n)$ is the category of \emph{monochromatic spectra}, denoted $\M_n$ and the category of complete objects are the $K(n)$-local spectra, $\Sp_\Kn$. For more details on monochromatic and $K(n)$-local spectra, see \cite{hovey-strickland_99}, and for the relationship to local duality, see \cite[Section 6.2]{barthel-heard-valenzuela_2018}. 

\begin{proposition}
    For any prime $p$ and non-negative integer $n$, there are equivalences 
    \[\M_n \simeq \Comod_{M_n\S}(\Sp_n) \text{ and } \SpKn\simeq \Contra_{M_n\S}(\Sp_n)\] 
    of symmetric monoidal stable $\infty$-categories. 
\end{proposition}
\begin{proof}
    This follows directly from \cref{thm:local-duality-co-contra}, as $L_n F(n)$ is compact in $\Sp_n$, making the pair $(\Sp_n, L_n F(n))$ a local duality context. 
\end{proof}

By \cref{cor:contra_unit_complete} we also have a description of $\SpKn$ as contramodules over the $K(n)$-local sphere $L_\Kn \S$.

\begin{corollary}
    \label{cor:Kn-local-spectra-as-contramodules}
    There is an equivalence of symmetric monoidal stable $\infty$-categories
    \[\SpKn \simeq \Contra_{L_\Kn \S}(\Sp_n).\]
\end{corollary}

\begin{remark}
    By \cite[6.3]{davis-lawson_2014} $L_\Kn\S$ has a commutative pro-dualizable presentation in terms of Moore spectra, and one can show that this gives an equivalent category of contramodules to the one determined by $M_n\S$. 
    %By \cite[2.2.1, 2.2.7]{li-zhang_2023} we know that there is an equivalence 
    % \[\Sp_\Kn \simeq \lim_j \Mod_{L_n F_j}(\Spn),\]
    % which by \cref{cor:Kn-local-spectra-as-contramodules} gives an equivalence 
    % \[\lim_j \Mod_{L_n F_j}(\Spn) \simeq \Mod_{\lim_j(L_n F_j \otimes -)}(\Sp_n),\]
    % meaning that the modules functor, as a functor of monads, commutes with limit. 
\end{remark} 

We also have a similar description of the objects coming from the finite chromatic filtration. 

\begin{proposition}
    For any prime $p$ and non-negative integer $n$, there are equivalences 
    \[\M_n^f \simeq \Comod_{M_n^f\S}(\Sp_n^f) \text{ and }\Sp_{T(n)}\simeq \Contra_{M_n^f\S}(\Sp_n^f)\] 
    of symmetric monoidal stable $\infty$-categories. 
\end{proposition}
\begin{proof}
    As the functor $M_n^f\: \Sp_n^f \to \M_n^f$ is a smashing colocalization, \cref{thm:dual-monoidal-BB} gives an equivalence $\M_n^f \simeq \Comod_{M_n^f \S}(\Sp_n^f).$ As there is an equivalence $\M_n^f \simeq \Sp_{T(n)}$ the claim of the result is then a formal consequence of \cref{thm:Positselski-duality-coidempotent}.
\end{proof}

As above, this gives by \cref{cor:contra_unit_complete} the following description.

\begin{corollary}
    \label{cor:Tn-local-spectra-as-contramodules}
    There is an equivalence of symmetric monoidal stable $\infty$-categories
    \[\Sp_{T(n)} \simeq \Contra_{L_{T(n)} \S}(\Sp_n).\]
\end{corollary}

\subsubsection*{Derived completion}
\label{ch2:ssec:derived-completion}

Let $R$ be a commutative noetherian ring and $I\subseteq R$ an ideal generated by a finite regular sequence. The $I$-adic completion functor 
\[C^I\: \Mod_R \to \Mod_R,\] 
defined by $C^I(M)=\lim_k M/I^k$ is neither a left nor right exact functor. However, by \cite[5.1]{greenlees-may_92} the higher right derived functors vanish. We denote the higher left derived functors of $C^I$ by $L^I_i$. An $R$-module $M$ is said to be \emph{$I$-adically complete} if the natural map $M\to C^I (M)$ is an isomorphism. It is said to be \emph{$L$-complete} if the natural map $M\to L_0^I(M)$ is an isomorphism. 

The map $M\to C^I(M)$ factors through $L_0^I(M)$, and the map $L_0^I(M)\to C^I(M)$ is always an epimorphism, but usually not an isomorphism. The full subcategory consisting of the $L$-complete modules form an abelian category $\Mod_R\Icomp$. The full subcategory of $I$-adically complete modules, $\Mod_R^\wedge$ is usually not abelian. 

The $I$-power torsion submodule of an $R$-module $M$ is defined to be 
\[T_I(M) := \{m \in M \mid I^k m = 0 \text{ for some } k\geq 0\}.\]
We say an $R$-module $M$ is $I$-power torsion if the natural map $T_I(M) \to M$ is an isomorphism. The full subcategory consisting of $I$-power torsion $R$-modules forms a Grothendieck abelian category, denoted $\Mod_R\Itors$. 

The object $R/I$ is compact in $\Der(R)$, which is a rigidly compactly generated symmetric monoidal stable $\infty$-category. Hence, $(\Der(R), R/I)$ is a local duality context. The associated local duality functors $\Gamma$ and $\Lambda$ coming can by \cite[3.16]{barthel-heard-valenzuela_2018} be identified with the total right derived functor $\mathbb{R}T_I$ and the total left derived functor $\mathbb{L}C^I$ respectively. By \cref{thm:Positselski-duality-coidempotent} we know that these are the cofree comodule functor and the free contramodule functor, hence we can conclude with the following. 

\begin{proposition}
    There are symmetric monoidal equivalences
    \[\Der(R)^{R/I-tors}\simeq \Comod_{\mathbb{R}T_I(R)} \text{ and } \Der(R)^{R/I-\mathrm{comp}}\simeq \Contra_{\mathbb{R}T_I(R)}.\]
\end{proposition}

As for the chromatic example described above, we also get by \cref{cor:contra_unit_complete} a description of derived complete modules in terms of contramodules over the completed unit. 

\begin{corollary}
    \label{cor:contra_derived_complete}
    There is an equivalence of symmetric monoidal stable $\infty$-categories 
    \[\Der(R)^{R/I-\mathrm{comp}} \simeq \Contra_{\mathbb{L} C^I(R)}.\]
\end{corollary}

\begin{remark}
    Interestingly, there are also descriptions of the category $\Mod_R\Icomp$ as a category of contramodules. In particular, $\Mod_R\Icomp$ is equivalent to the category of (classical) contramodules over the pro-finite completed ring $C^I(R) = \lim_k R/I^k$, see \cite[Section 2.2]{positselski_2022_contramodules}. Furthermore, the category $\Der(R)^{R/I-\mathrm{comp}}$ is by \cite[3.7(1)]{barthel-heard-valenzuela_2020} equivalent to the right completion of the derived category of $\Mod_R\Icomp$. 

    Combining these facts with \cref{cor:contra_derived_complete} we then obtain an equivalence
    \[\Der(\Contra_{C^I(R)}) \simeq \Contra_{\mathbb{L} C^I(R)}(\Der(R)),\]
    which connects our $\infty$-categorical notion of contramodules to Positselski's classical theory via the total left derived functor. 

    % There is a similar story for comodules. The category $\Mod_R\Itors$ is equivalent to the category of comodules over a coalgebra $C$ such that $C^I(R) = \Hom(C,k)$. Furthermore, The category $\Der(R)^{R/I-\mathrm{tors}}$ is by \cite[3.7(2)]{barthel-heard-valenzuela_2020} equivalent to $\Der(\Mod_R\Itors)$, the derived category $I$-power torsion modules. 

    % Hence, we obtain an equivalence 
    % \[\Der(\Comod_{C})\]
    
    % This makes the above example into an example of the derived co-contra-correspondence, see for example \cite{positselski_2016}. 
\end{remark}

\subsection*{Acknowledgements and personal remarks}

The contents of this paper go back to one of the first ideas I had at the beginning of my PhD. I had my two favorite mathematical hammers---local duality and the monoidal Barr--Beck theorem---and was trying to see if these were really one and the same tool. Local duality consists of three parts: local objects, torsion objects and complete objects. The core idea came from the fact that the local objects are modules over an idempotent algebra, and I thus wanted a similar description of the other two parts. Drew Heard's guidance led me to a dual monoidal Barr--Beck result, checking off the torsion part. I got the first hints of the last piece after an email correspondence with Marius Nielsen, where we discussed a local duality type statement for mapping spectra. The solution clicked into place during a research visit to Aarhus University. During my stay Sergey Arkhipov gave two talks on contramodules, for completely unrelated reasons, and I immediately knew this was the last piece of the puzzle. Greg Stevenson taught me some additional details, solidifying my ideas, which led me to conjecture one of the main results of the present paper during my talk in their seminar. The crowd nodded in approval, thus, being satisfied I knew the answer, I naturally spent almost two years not writing it up.

I want to thank all of the people mentioned above for their insights and pathfinding skills, without which this project would still have been a rather simple-minded idea in the optimistic brain of a young PhD student. 

After the first version of this paper was made public, I continued to try to understand how to make $\Lambda \1$ naturally show up in the Positselski duality. I am indebted to Florian Riedel for teaching me things about coalgebras, and for helping me solidify several related ideas. 

\printbibliography{}

\textbf{Torgeir Aamb\o:} Department of Mathematical Sciences, Norwegian University of Science and Technology, Trondheim
 
\end{document}